\documentclass[11pt, letterpaper]{article}
\usepackage[margin=3.4cm]{geometry}
\usepackage[numbers]{natbib}
\RequirePackage[colorlinks,citecolor=blue,urlcolor=blue, linkcolor=blue]{hyperref}
\usepackage{comment}
\usepackage{lineno}
\usepackage{amsmath, amsthm, amssymb, mathtools, bm, bbm, mathrsfs}
\usepackage{graphicx} 
\usepackage[capitalise]{cleveref}
\Crefname{assumption}{Assumption}{Assumptions}
\usepackage{chngcntr}
\usepackage{apptools}
\AtAppendix{\counterwithin{definition}{section}}
\AtAppendix{\counterwithin{lemma}{section}}
\AtAppendix{\counterwithin{figure}{section}}
\usepackage{caption}
\usepackage{booktabs, multirow}
\usepackage{enumerate}
\usepackage{algpseudocode, algorithm}
\usepackage[enable]{easy-todo}

\usepackage{etoolbox}

\newtheorem{theorem}{Theorem}

\newtheorem{proposition}{Proposition}
\newtheorem{definition}{Definition}
\newtheorem{lemma}{Lemma}

\newtheorem{corollary}{Corollary}

\theoremstyle{definition}
\newtheorem{remark}{Remark}

\theoremstyle{definition}

\makeatletter
\newcommand{\longdash}[1][2em]{\makebox[#1]{$\m@th\smash-\mkern-7mu\cleaders\hbox{$\mkern-2mu\smash-\mkern-2mu$}\hfill\mkern-7mu\smash-$}}
\makeatother

\makeatletter
\patchcmd{\NAT@test}{\else \NAT@nm}{\else \NAT@nmfmt{\NAT@nm}}{}{}
\DeclareRobustCommand\citepos
  {\begingroup
   \let\NAT@nmfmt\NAT@posfmt \NAT@swafalse\let\NAT@ctype\z@\NAT@partrue
   \@ifstar{\NAT@fulltrue\NAT@citetp}{\NAT@fullfalse\NAT@citetp}}

\let\NAT@orig@nmfmt\NAT@nmfmt
\def\NAT@posfmt#1{\NAT@orig@nmfmt{#1's}}
\makeatother

\newcommand{\omitskip}{\kern-\arraycolsep}

\DeclareMathOperator*{\D}{\mathcal{D}}

\renewcommand{\P}{\mathbb{P}}

\DeclareMathOperator{\E}{\mathbb{E}}

\newcommand*\dd{\mathop{}\!\mathrm{d}}

\newcommand{\distconvto}{\rightarrow_{d}}

\newcommand{\gam}{\mathsf{Ga}}
\newcommand{\mult}{\mathsf{Mult}}

\begin{document}

\title{Chernoff-type Concentration of Empirical Probabilities in Relative Entropy}

\author{ {\bf F.\ Richard Guo} \\Department of Statistics \\ University of Washington\\
\tt{ricguo@uw.edu}\\ \and 
{\bf Thomas~S.~Richardson} \\Department of Statistics \\ University of Washington\\ \tt{thomasr@u.washington.edu} \\
}

\maketitle

\begin{abstract}
We study the relative entropy of the empirical probability vector with respect to the true probability vector in multinomial sampling of $k$ categories, which, when multiplied by sample size $n$, is also the log-likelihood ratio statistic. We generalize a recent result and show that the moment generating function of the statistic is bounded by a polynomial of degree $n$ on the unit interval, uniformly over all true probability vectors. We characterize the family of polynomials indexed by $(k,n)$ and obtain explicit formulae. Consequently, we develop Chernoff-type tail bounds, including a closed-form version from a large sample expansion of the bound minimizer. Our bound dominates the classic method-of-types bound and is competitive with the state of the art. We demonstrate with an application to estimating the proportion of unseen butterflies.
\end{abstract}

\section{Introduction}
Consider a multinomial experiment on an alphabet of size $k \geq 2$
\begin{equation}
(X_1, \dots, X_k) \sim \mult(n; (p_1, \dots, p_k)),
\end{equation}
where $(p_1, \dots, p_k)$ belongs to the unit simplex $\Delta^{k-1}$. The empirical measure is identified with the probability vector $(\hat{p}_1, \dots, \hat{p}_k) = (X_1/n, \dots, X_k/n)$. We are interested in its entropy relative to the true probability vector $p$, namely
\begin{equation}
\D(\hat{p} \| p) = \sum_{i=1}^{k} \hat{p}_i \log (\hat{p}_i / p_i),
\end{equation}
where conventions $0 \cdot \log(0) = 0$ and $0 \cdot \log(0 / 0) = 0$ are adopted. The quantity $\D(\hat{p} \| p)$ is also known as the Kullback-Leibler divergence of $p$ from $\hat{p}$. By the law of large numbers, $\D(\hat{p} \| p) \rightarrow 0$ as $n \rightarrow \infty$ almost surely.

Note that 
\begin{equation*}
n \D(\hat{p} \| p) = \sum_{i=1}^{k} X_i \log \frac{\hat{p}_i}{p_i} = \log \frac{\binom{n}{X_1,\dots,X_k} \prod_{i=1}^{k} \hat{p}_i^{X_i}}{\binom{n}{X_1,\dots,X_k} \prod_{i=1}^{k} p_i^{X_i}}
\end{equation*}
is also the log-likelihood ratio statistic (without the usual extra factor of 2). By standard asymptotic arguments (see, e.g., \citet[Example 16.1]{van2000asymptotic}), for fixed $k$ and $n \rightarrow \infty$, it holds that
\begin{equation} \label{eqs:asymp}
n \D(\hat{p} \| p) \ \distconvto\  \chi^2_{k-1} / 2 =_{d} \gam((k-1)/2, 1),
\end{equation}
which is a gamma distribution with shape $(k-1)/2$ and rate one. 

We are interested in upper bounding the probability that $n \D(\hat{p} \| p)$ exceeds a given threshold. Tail bounds of this type are of interest to many problems in probability, statistics and machine learning, including Sanov's theorem in large deviations \citep[\S 11.4]{cover2012elements}, goodness-of-fit tests \citep{cressie1984multinomial,jager2007goodness}, construction of non-asymptotic confidence regions \citep{chafai2009confidence,malloy2020optimal} and the performance guarantee of various learning algorithms \citep{vinayak2019maximum, nowak2019tighter}.

The classic bound of this type is 
\begin{equation} \label{eqs:MoT}
\P\left(n \D(\hat{p} \| p) > t \right) \leq \exp(-t) \binom{n + k - 1}{k - 1}, \quad t > 0
\end{equation}
obtained by the ``method of types'' \citep[Lemma II.1]{csiszar1998method}. For fixed $k$ and $t$, this bound is asymptotically tight as $n \rightarrow \infty$, in the sense that the exponent $\exp(-t)$ matches the rate of the asymptotic gamma distribution in \cref{eqs:asymp}. Nevertheless, the bound above is far from optimal. There are recent developments in the literature that provide sharper results. In particular, \citet{mardia2018concentration} and \citet{agrawal2019concentration} provide significant improvements over the method-of-types result by gaining tighter control for the binomial case ($k=2)$, and a reduction from multinomial ($k>2$) to binomial, although their approaches are different. Additionally, bounds on the moments of $\D(\hat{p} \| p)$ have been studied; see  \citet{jiao2017maximum,mardia2018concentration,paninski2003estimation}.

On a side note, by Pinsker's inequality, a tail bound on the relative entropy implies a bound on the total variation. For bounds on the latter, see also \citet[Appendix A.6]{vaart1996weak}, \citet{devroye1983equivalence} and \citet{biau2005asymptotic}.

\section{Bounding the moment generating function}
In a vein similar to that of \citet{agrawal2019concentration}, we develop bounds with Chernoff's method, a classic workhorse for deriving exponential tail bounds;  see, e.g., \citet[\S 2.3]{vershynin2018high}. 
The key is to upper bound the moment generating function (MGF) of $n \D(\hat{p} \| p)$, which is defined as 
\begin{equation}
\varphi_{k,n}(\lambda, p) := \E \exp\left(\lambda n \D(\hat{p} \| p)\right),
\end{equation}
where the expectation is taken over $\mult(n, p=(p_1, \dots, p_k))$. 

It follows that 
\begin{equation} \label{eqs:varphi}
\begin{split}
\varphi_{k,n}(\lambda, p) &= \sum_{X_1, \dots, X_k} \binom{n}{X_1,\dots,X_k} \prod_{i=1}^{k} p_i^{X_i} \left \{ \frac{\binom{n}{X_1,\dots,X_k} \prod_{i=1}^{k} \hat{p}_i^{X_i}}{\binom{n}{X_1,\dots,X_k} \prod_{i=1}^{k} p_i^{X_i}} \right \}^{\lambda} \\
&= \sum_{X_1, \dots, X_k} \binom{n}{X_1,\dots,X_k} \left\{ \prod_{i=1}^{k} \hat{p}_i^{X_i} \right \}^{\lambda} \left\{ \prod_{i=1}^{k} p_i^{X_i} \right\}^{1-\lambda},
\end{split} 
\end{equation}
where $X_1, \dots, X_k$ are non-negative integers that sum to $n$.

\begin{definition} \label{def:G}
For $k \geq 1$, $n \geq 1$, $p \in \Delta^{k-1}$ and $\lambda \in [0,1]$, define
\begin{equation} \label{eqs:G-kn}
G_{k,n}(\lambda, p) := \sum_{X_1,\dots,X_k} \binom{n}{X_1,\dots,X_k} \prod_{i=1}^{k} \left[\lambda X_i / n + (1-\lambda) p_i \right]^{X_i},
\end{equation}
where the summation is over non-negative integers that sum to $n$. 
\end{definition}
By definition, $G_{k,n}(\lambda, p)$ is a polynomial in $\lambda$ of degree at most $n$. For the trivial case of $k=1$, it is easy to see that $G_{1,n}(\lambda) \equiv 1$. 

The multinomial probability in \cref{eqs:varphi} is log-concave in $(p_1, \dots, p_k)$. For $0\leq \lambda \leq 1$, by Jensen's inequality, we have 
\begin{equation*}
\varphi_{k,n}(\lambda, p) \leq G_{k,n}(\lambda, p), \quad p \in \Delta^{k-1}.
\end{equation*}
The obvious obstacle here is to obtain a bound on the RHS that does not depend on the true probability vector $p$.

\subsection{Family of $G_{k,n}(\lambda)$}
First comes a surprising fact noticed by \citet{agrawal2019concentration} in the $k=2$ case.
\begin{proposition} \label{prop:p-independent}
$G_{k,n}(\lambda, p)$ does not depend on $p = (p_1, \dots, p_k)$. 
\end{proposition}
\begin{proof}
This is true for $k=1$. Fix any $k \geq 2$, we prove by induction on $n$ that $G_{k,n}(\lambda, p)$ does not depend on $p$. For the base case, 
\begin{equation*}
G_{k,1}(\lambda, p) = \sum_{i=1}^{k} \left(\lambda + (1-\lambda) p_i \right) = k \lambda + 1 - \lambda,
\end{equation*}
which does not depend on $p$. 

Suppose $G_{k,m}(\lambda,p) \equiv G_{k,m}(\lambda)$ for $m \leq n-1$. We now show that $G_{k,n}(\lambda, p)$ does not depend on $p$. Since $p_k = 1 - p_1 - \dots - p_{k-1}$, it suffices to verify that $\partial G_{k,n}(\lambda, p)/ \partial p_i \equiv 0$ for $i=1,\dots,k-1$. Further, by symmetry, it suffices to show $\partial G_{k,n}(\lambda, p)/ \partial p_1 \equiv 0$. Replacing $p_k$ with $(1-p_1 - \dots - p_{k-1})$, we have
\begin{multline*}
G_{k,n}(\lambda, p) = \sum_{X_1,\dots,X_k} \binom{n}{X_1,\dots,X_k} \prod_{j=2}^{k-1} \left[\lambda X_j / n + (1-\lambda) p_j \right]^{X_j}  \\
\times \left[\lambda X_1 / n + (1-\lambda) p_1 \right]^{X_1} \left[\lambda X_k / n + (1-\lambda)(1-p_1-\dots-p_{k-1})  \right]^{X_k},
\end{multline*}
and 
\begin{multline*}
\frac{\partial G_{k,n}(\lambda, p)}{\partial p_1} = \sum_{X_1,\dots,X_k} \binom{n}{X_1,\dots,X_k} \prod_{j=2}^{k-1} \left[\lambda X_j / n + (1-\lambda) p_j \right]^{X_j} \\
\bigg\{ (1-\lambda) X_1 \left[\lambda X_1 / n + (1-\lambda) p_1 \right]^{X_1-1} \left[\lambda X_k / n + (1-\lambda) p_k \right]^{X_k} \\
-(1-\lambda) X_k \left[\lambda X_1 / n + (1-\lambda) p_1 \right]^{X_1} \left[\lambda X_k / n + (1-\lambda) p_k \right]^{X_k-1} \bigg\}.
\end{multline*}
Hence, it suffices to show
\begin{multline*}
\sum_{X_1,\dots,X_k} \binom{n}{X_1,\dots,X_k} \prod_{j=2}^{k-1} \left[\lambda X_j / n + (1-\lambda) p_j \right]^{X_j} \\
\times X_1 \left[\lambda X_1 / n + (1-\lambda) p_1 \right]^{X_1-1} \left[\lambda X_k / n + (1-\lambda) p_k \right]^{X_k}  \equiv \sum_{X_1,\dots,X_k} \binom{n}{X_1,\dots,X_k} \\
\times \prod_{j=2}^{k-1} \left[\lambda X_j / n + (1-\lambda) p_j \right]^{X_j} 
X_k \left[\lambda X_1 / n + (1-\lambda) p_1 \right]^{X_1} \left[\lambda X_k / n + (1-\lambda) p_k \right]^{X_k-1}.
\end{multline*}
We first simplify the LHS. Clearly, those summands with $X_1=0$ are zero and can be dropped. For $X_1 \geq 1$, $X_1 \binom{n}{X_1,\dots, X_k} = n \binom{n-1}{X_1-1,X_2,\dots,X_k}$. Let $\lambda' := \lambda (n-1)/n$. For $j=2,\dots,k$, by setting $p_j' := \frac{1-\lambda}{1-\lambda'} p_j < p_j$, we have
\begin{equation*}
\lambda X_j / n + (1-\lambda) p_j  = \lambda' X_j / (n-1) + (1-\lambda') p_j'.
\end{equation*}
Further, letting $p_1' := 1 - \sum_{j=2}^{k}p_j'$ it is easy to see that
\begin{equation*}
\lambda' \frac{X_1-1}{n-1} + (1-\lambda') p_1' = \lambda \frac{X_1}{n} + (1-\lambda) p_1.
\end{equation*}
Therefore, by introducing $X_1' = X_1 - 1$, we have
\begin{equation*}
\begin{split}
\text{LHS} &= n \sum_{X_1',X_2,\dots,X_k} \binom{n-1}{X_1',X_2,\dots,X_k} \left[\lambda' X_1' / (n-1) + (1-\lambda') p_1' \right]^{X_1'} \\
& \quad \quad \times \prod_{j=2}^{k} \left[\lambda' X_j / (n-1) + (1-\lambda') p_j' \right]^{X_j} \\
&= n G_{k,n-1}(\lambda', p'),
\end{split}
\end{equation*} 
where the summation is over non-negative integers $X_1', X_2, \dots, X_k$ summing to $n-1$. 
For the RHS, similarly, let $q_j' = \frac{1-\lambda}{1-\lambda'} p_j$ for $j=1,\dots,k-1$ and $q_{k}' = 1 - \sum_{j=1}^{k-1} q_j'$. With $X_k' = X_k - 1$, it follows that
\begin{equation*}
\begin{split}
\text{RHS} &= n \sum_{X_1,\dots,X_{k-1},X_k'} \binom{n-1}{X_1,\dots,X_{k-1},X_k'} \prod_{j=1}^{k-1} \left[\lambda' X_j / (n-1) + (1-\lambda') q_j' \right]^{X_j} \\
& \quad \quad \times \left[\lambda' X_k' / (n-1) + (1-\lambda') q_k' \right]^{X_k'} \\
&= n G_{k, n-1}(\lambda', q').
\end{split}
\end{equation*}
Finally, by the induction hypothesis, 
\begin{equation*}
\text{LHS} = n G_{k,n-1}(\lambda', p') = n G_{k,n-1}(\lambda', q') = \text{RHS}.
\end{equation*}
\end{proof}

In view of this fact, we shall write $G_{k,n}(\lambda)$ in place of $G_{k,n}(\lambda, p)$. The set of polynomials $\{G_{k,n}(\lambda)\}$ are characterized by the following recurrence. 

\begin{proposition} \label{prop:recurrence}
For $0 \leq \lambda \leq 1$, it holds that 
\begin{equation} \label{eqs:recurrence}
G_{k,n}(\lambda) = G_{k-1,n}(\lambda) + \lambda G_{k,n-1}\left(\frac{n-1}{n} \lambda \right), \quad k \geq 2, \  n \geq 1
\end{equation}
with $G_{1,n}(\lambda) \equiv 1$ and  $G_{k,0}(\lambda) := 1$.
\end{proposition}
By \cref{prop:p-independent}, we have the freedom to choose $p$ in the definition to evaluate $G_{k,n}(\lambda)$. In particular, by choosing $p_k=0$ and $p_1 + \dots + p_{k-1} = 1$, we can decompose $G_{k,n}(\lambda)$ into $G_{k-1,n}(\lambda)$ and a remainder. By a similar manipulation used in the previous proof, the remainder can be expressed in terms of $G_{k,n-1}$. We leave the detailed proof to the Appendix.

\begin{theorem}  \label{thm:polynomial}
For $k \geq 2$, $n \geq 0$ and $0\leq \lambda \leq 1$, it holds that 
\begin{equation} \label{eqs:polynomial}
G_{k,n}(\lambda) = \sum_{m=0}^{n} \frac{n!}{n^m (n-m)!}\binom{m+k-2}{k-2} \lambda^{m}.
\end{equation}
\end{theorem}
\begin{proof}
We prove by induction. For the base case, the formula gives $G_{k,0}(\lambda) \equiv 1$ for $k \geq 2$, which matches the value imposed by \cref{prop:recurrence}. 

First, supposing the formula holds for $G_{2,n-1}$, we show that it also holds for $G_{2,n}$. By \cref{prop:recurrence}, it is easy to check that
\begin{equation*}
\begin{split}
G_{2,n}(\lambda) &= G_{1,n}(\lambda) + \lambda G_{2,n-1}(\lambda(n-1)/n) \\
&= 1 + \sum_{m=0}^{n-1} \frac{(n-1)!}{n^m (n-m-1)!} \lambda^{m+1} = \sum_{m=0}^{n} \frac{n!}{n^m (n-m)!} \lambda^m.
\end{split}
\end{equation*}

Now, for any $k \geq 3$ and $n \geq 1$, suppose the formula holds for $G_{k-1,n}$ and $G_{k,n-1}$. We show that it also holds for $G_{k,n}$. By \cref{prop:recurrence}, we have
\begin{equation*}
\begin{split}
G_{k,n}(\lambda) &= G_{k-1,n}(\lambda) + \lambda G_{k,n-1}(\lambda(n-1)/n) \\
&= \sum_{m=0}^{n} \frac{n!}{n^m (n-m)!}\binom{m+k-3}{k-3} \lambda^{m} + \sum_{m=0}^{n-1} \frac{(n-1)!}{n^m (n-m-1)!}\binom{m+k-2}{k-2} \lambda^{m+1} \\
&= \sum_{m=0}^{n} \frac{n!}{n^m (n-m)!}\binom{m+k-3}{k-3} \lambda^{m} + \sum_{m=1}^{n} \frac{n!}{n^{m} (n-m)!}\binom{m+k-3}{k-2} \lambda^{m} \\
&= \sum_{m=0}^{n} \frac{n!}{n^m (n-m)!}\binom{m+k-2}{k-2} \lambda^{m},
\end{split}
\end{equation*}
where in the last step the addition formula $\binom{n}{l} = \binom{n-1}{l} + \binom{n-1}{l-1}$ is used \citep[\S 5.1]{gkp-book}.  
\end{proof}

\begin{remark}
$G_{k,n}(\lambda)$ can be rewritten as
\begin{equation*}
G_{k,n}(\lambda) = \sum_{m=0}^n \binom{n}{m} (k-1)^{(m)} (\lambda  / n)^m,
\end{equation*}
where $x^{(m)} = x (x+1) \dots (x+m-1)$ is the rising factorial. 
\end{remark}

\begin{lemma} \label{eqs:Gkn-deriv-form}
For $k \geq 2$, 
\begin{equation*}
G_{k,n}(\lambda) = \frac{1}{(k-2)!} \frac{\dd^{k-2}}{\dd \lambda^{k-2}} \left(\lambda^{k-2} G_{2,n}(\lambda) \right).
\end{equation*}
\end{lemma}
\begin{proof}
\begin{equation*}
\begin{split}
\frac{1}{(k-2)!} \frac{\dd^{k-2}}{\dd \lambda^{k-2}} \left(\lambda^{k-2} G_{2,n}(\lambda) \right) &= \sum_{m=0}^{n} \binom{n}{m} \frac{m!}{n^m (k-2)!} \frac{\dd^{k-2}}{\dd \lambda^{k-2}} \lambda^{m+k-2} \\
&= \sum_{m=0}^{n} \binom{n}{m} \frac{m!}{n^m (k-2)!} \frac{(m+k-2)!}{m!} \lambda^m \\
&= \sum_{m=0}^{n} \binom{n}{m} (k-1)^{(m)} (\lambda / n)^m = G_{k,n}(\lambda).
\end{split}
\end{equation*}
\end{proof}

\begin{remark}
For $k \geq 2$, $G_{k,n}(\lambda)$ is not a moment generating function of some distribution. To see this, suppose $G_{k,n}(\lambda)$ is the MGF of some random variable $Y$. Since $G_{k,n}(\lambda)$ is a polynomial of degree $n$, then $\E Y^{2n} = G_{k,n}^{(2n)}(0) = 0$, which implies $Y$ is zero almost surely. However, the MGF of zero is identically one. 
\end{remark}

A few polynomials $G_{k,n}(\lambda)$ are listed in \cref{tab:poly}.

\begin{table}[!htb]
\caption{Polynomials $G_{k,n}(\lambda)$}
\label{tab:poly}
\begin{center}
\begin{tabular}{@{}lcccc@{}}
\toprule
& $n=1$ & $n=2$ & $n=3$ & $n=4$ \\ \midrule
$k=2$ & $1+\lambda$ & $1+\lambda + \frac{1}{2}\lambda^2$ &  $1+\lambda + \frac{2}{3} \lambda^2 + \frac{2}{9} \lambda^3 $ &  $1 + \lambda + \frac{3}{4} \lambda^2 + \frac{3}{8} \lambda^3 + \frac{3}{32} \lambda^4$ \\[3pt]
$k=3$ & $1+ 2\lambda$ & $1+ 2\lambda + \frac{3}{2}\lambda^2$ &  $1+ 2\lambda + 2\lambda^2 + \frac{8}{9} \lambda^3 $ &  $1 + 2\lambda + \frac{9}{4} \lambda^2 + \frac{3}{2} \lambda^3 + \frac{15}{32} \lambda^4$ \\[3pt]
$k=4$ & $1+ 3\lambda$ & $1+ 3\lambda + 3\lambda^2$ &  $1+ 3\lambda + 4\lambda^2 + \frac{20}{9} \lambda^3 $ &  $1 + 3\lambda + \frac{9}{2} \lambda^2 + \frac{15}{4} \lambda^3 + \frac{45}{32} \lambda^4$ \\ \bottomrule
\end{tabular}
\end{center}
\end{table}

\subsection{Asymptotic properties}
We consider the asymptotic behaviors of $G_{k,n}(\lambda)$, which can inform how well it captures the right dependence on $k$ and $n$. 

\subsubsection{$n \rightarrow \infty$ under fixed $k$}
\begin{lemma} \label{lem:monotone}
For $k \geq 2$, $G_{k,n}(\lambda)$ increases in $n$. 
\end{lemma}
\begin{proof}
By \cref{thm:polynomial}, it suffices to show that 
\begin{equation*}
\frac{n!}{n^m (n-m)!} \geq \frac{(n-1)!}{(n-1)^m (n-m-1)!}
\end{equation*}
for $m=0,\dots,n$. By canceling factors from both sides, this is equivalent to $(1- \frac{1}{n})^m \geq 1 - \frac{m}{n}$, which holds by Bernoulli's inequality. 
\end{proof}

\begin{proposition} \label{prop:large-n-limit}
For $0 \leq \lambda <1$ and any fixed $k \geq 2$, we have
\begin{equation}
G_{k,n}(\lambda) \nearrow G_{k,\infty}(\lambda) := (1-\lambda)^{-(k-1)} , \quad \text{as } n \rightarrow \infty.
\end{equation}
\end{proposition}
\begin{proof}
For $k=2$ and $\lambda \in [0,1)$, 
\begin{equation*}
G_{2,n}(\lambda) = \sum_{m=0}^{n} \frac{n!}{n^m (n-m)!} \lambda^m \leq \sum_{m=0}^{n} \lambda^m \rightarrow \frac{1}{1-\lambda},
\end{equation*}
where we used 
\begin{equation*}
\frac{n!}{n^m(n-m)!} = \frac{n \times (n-1) \times \dots (n-m+1)}{n \times  \dots \times n} \leq 1.
\end{equation*}
Further, by \cref{lem:monotone}, $G_{2,n}(\lambda)$ must converge as $n \rightarrow \infty$ for $\lambda \in [0, 1)$. Suppose the limit is $G_{2,\infty}(\lambda)$. Clearly, $G_{2,\infty}(\lambda) = \lim_{n} G_{2,n}(\lambda) = \sup_{n} G_{2,n}(\lambda)$ is lower-semicontinuous. 
Taking limits on both sides of \cref{eqs:recurrence}, we have
\begin{equation*}
G_{2,\infty}(\lambda) = 1 + \lambda G_{2,\infty}(\lambda^{-}),
\end{equation*}
where we note $\frac{n-1}{n} \lambda \nearrow \lambda$. Meanwhile, by \cref{thm:polynomial}, $G_{2,n}(\lambda)$ is increasing in $\lambda$. Hence, we have $G_{2,\infty}(\lambda^{-}) = G_{2,\infty}(\lambda)$ by lower-semicontinuity and monotonicity of $G_{2,\infty}(\lambda)$. It follows that $G_{2,\infty} = (1-\lambda)^{-1}$. Applying the same reasoning to $k=3$, we have
\begin{equation*}
G_{3,\infty}(\lambda) = G_{2,\infty}(\lambda) + \lambda G_{3,\infty}(\lambda),
\end{equation*}
and hence $G_{3,\infty} = (1-\lambda)^{-2}$. Iterating this process, we get $G_{k,n}(\lambda) \nearrow (1-\lambda)^{-(k-1)}$ for $\lambda \in [0,1)$ and $k \geq 2$. 
\end{proof}

Note that $G_{k,\infty}(\lambda) = (1-\lambda)^{-(k-1)}$ is the moment generating function of $\gam(k-1,1)$. Further, $n \D(\hat{p} \| p) \distconvto \gam((k-1)/2, 1)$. This means, for fixed $k$ and $n \rightarrow \infty$, $G_{k,n}(\lambda)$ is asymptotically tight in the exponent (rate parameter of gamma), but loose by a factor of $2$ in the polynomial term (shape parameter of gamma).

\subsubsection{$k \rightarrow \infty$ under fixed $n$}
In the following, for two sequences $a_\nu$ and $b_\nu$ that are positive for large enough $\nu$, we write $a_\nu \gtrsim b_\nu$ (or $b_\nu \lesssim a_\nu$) as $\nu \rightarrow \infty$ if there exists $c>0$ such that $a_\nu \geq c b_\nu$ for all $\nu$ large enough. If $a_\nu \lesssim b_\nu$ and $a_\nu \gtrsim b_\nu$, we write $a_\nu \asymp b_\nu$ as $\nu \rightarrow \infty$. 

\begin{proposition} \label{prop:G-log-k}
For fixed $0 < \lambda \leq 1$ and $n \geq 1$, as $k \rightarrow \infty$ we have
\begin{equation}
\log G_{k,n}(\lambda) \asymp n \log k.
\end{equation}
\end{proposition}
\begin{proof}
By \cref{thm:polynomial}, for fixed $n$ and $\lambda$, the leading term should be the largest term of $\{\binom{m+k-2}{k-2}: 0 \leq m \leq n\}$, which is when $m=n$. To see this, note that $\binom{m+k-2}{k-2} = (k-1) \times k \times \dots \times (k+m-2) / m!$. It then follows from $\log \binom{n+k-2}{k-2} \asymp n \log k$. 
\end{proof}

The following shows that, as $k \rightarrow \infty$, the logarithmic dependence on $k$ for an upper bound on the logarithm of MGF is also necessary. 
\begin{proposition}
Suppose $H_{k,n}(\lambda) \geq \varphi_{k,n}(\lambda; p)$ for all $p$ and all $\lambda \in (0,1)$. For fixed $0 < \lambda \leq 1$ and $n \geq 1$, we have lower bound $\log H_{k,n}(\lambda) \gtrsim \lambda n \log k$ as $k \rightarrow \infty$. 
\end{proposition}
\begin{proof}
Let $p = (1/k, \dots, 1/k)$. It follows from \cref{eqs:varphi} that 
\begin{equation*}
\varphi_{k,n}(\lambda, p) = n^{-\lambda n} \sum_{X_1, \dots, X_k} \binom{n}{X_1, \dots, X_k} \prod_{i=1}^{k} \left(\frac{X_i^{\lambda}}{k^{1-\lambda}}\right)^{X_i}.
\end{equation*}
We claim that $\varphi_{k,n}(\lambda, p) \asymp k^{\lambda n}$. Consider the configurations of $(X_1, \dots, X_k)$ such that $n$ of them are one and the rest are zero. As $k \rightarrow \infty$, the sum over these configurations becomes
\begin{equation*}
n^{-\lambda n} \binom{k}{n} n! \left( \frac{1}{k^{1-\lambda}} \right)^{n} \gtrsim k^{\lambda n}.
\end{equation*}
\end{proof}

\begin{remark}
\citet{agrawal2019concentration} uses the upper bound $G_{2,\infty}(\lambda)$ on $G_{2,n}(\lambda)$ to further bound $G_{k,n}(\lambda)$ for $k > 2$, by appealing to the chain rule of relative entropy \citep[\S 2.5]{cover2012elements}. This leads to the following bound:
\begin{equation}
\varphi_{k,n}(\lambda) \leq (1 - \lambda)^{-(k-1)} = G_{k,\infty}(\lambda) \quad (0 \leq \lambda < 1).
\end{equation}
However, observe that for fixed $n$ and large $k$, the logarithm of the above bound above grows \emph{linearly} in $k$. In contrast, as we have shown via a direct approach, the bound $\log G_{k,n}(\lambda)$ has the right logarithmic dependence.
\end{remark}

\section{Chernoff bound}
To highlight the dependence on $(k,n)$, let $\hat{p}_{k,n}$ denote the empirical probability vector under $k$ categories and $n$ samples. For any $\lambda \in [0,1]$, we have
\begin{equation} \label{eqs:lambda}
\P\left(n \D(\hat{p}_{k,n} \| p) > t\right) \leq \exp(-\lambda t) G_{k,n}(\lambda).
\end{equation}

\subsection{Exact bound} 
Minimizing the above over $\lambda \in [0,1]$ yields the tightest bound. 
\begin{theorem} \label{thm:chernoff}
For $k \geq 2$, $n \geq 1$, let $\hat{p}_{k,n}$ be the empirical probability vector from $\mult(p, n)$ for $p \in \Delta^{k-1}$. For $t > 0$, it holds that 
\begin{equation} \label{eqs:chernoff}
\P\left( n \D(\hat{p}_{k,n} \| p) > t \right) \leq \min_{\lambda \in [0,1]} \exp(-\lambda t) G_{k,n}(\lambda).
\end{equation}
\end{theorem}

\begin{proposition} \label{prop:meaningful}
The bound in \cref{thm:chernoff} is meaningful ($\text{RHS} < 1$) if $t > \min(\log G_{k,n}(1), k-1)$. 
\end{proposition}
\begin{proof}
Let $f_{k,n}(\lambda, t) := \exp(-\lambda t) G_{k,n}(\lambda)$. Let $\psi_{k,n}(t) := \min_{\lambda \in [0,1]} f(\lambda, t)$ be the RHS of \cref{eqs:chernoff}. First, suppose $t > \min(\log G_{k,n}(1), k-1)$ and we show that $\psi_{k,n}(t)<1$. Clearly, either $t > \log G_{k,n}(1)$ or $t > k-1$. If $t > \log G_{k,n}(1)$, then $\psi_{k,n}(t) \leq f_{k,n}(1, t) = \exp(-t) G_{k,n}(1) < 1$. If $t > k-1$, $\psi_{k,n}(t) \leq \psi_{k, \infty}(t)$ by \cref{prop:large-n-limit}. One can show that 
\begin{equation*}
\psi_{k,\infty}(t) = \begin{cases} 0, &\quad t \leq k-1 \\
\exp(k-1-t) \left(\frac{t}{k-1}\right)^{k-1}, &\quad t > k-1 \end{cases}.
\end{equation*}
Writing $t = k - 1 + \delta$ for $\delta > 0$, it follows that 
\begin{equation*}
\begin{split}
\psi_{k,n}(t) \leq \psi_{k,\infty}(k-1 +\delta) &= \exp\left\{(k-1)\log(1 + \frac{\delta}{k-1}) - \delta \right\} < 1
\end{split}
\end{equation*}
by $\log(1+x) < x$ for $x>0$. 
\end{proof}

We have verified that the converse also holds at least for $k \leq 500$. 

Let $\lambda_{k,n}(t)$ be the minimizer in \cref{thm:chernoff}. Unfortunately, in general, $\lambda_{k,n}(t)$ does not permit a closed-form solution. In fact, finding $\lambda_{k,n}(t)$ is a non-convex problem and $\exp(-\lambda t) G_{k,n}(\lambda)$ can have more than one local minima on the unit interval. In the following, we develop a simple closed-form approximation to $\lambda_{k,n}(t)$ that leads to a bound that is only slightly looser than \cref{thm:chernoff}, when $n$ is relatively large compared to $k$.

\subsection{Large $n$ expansion of the minimizer}
By \cref{prop:large-n-limit}, when $n \rightarrow \infty$ we have
\begin{equation*}
\exp(-\lambda t) G_{k,n}(\lambda) \rightarrow \exp(-\lambda t) (1-\lambda)^{-(k-1)} = e^{-\lambda t -(k-1) \log (1-\lambda)}.
\end{equation*}
Note that $\lambda \mapsto -\lambda t -(k-1) \log (1-\lambda)$ is convex. The previous display is uniquely minimized at 
\begin{equation} \label{eqs:lambda-k-infty}
\lambda_{k,\infty}(t) = 1 - \frac{k-1}{t}, \quad \text{for $t>k-1$}.
\end{equation}
Plugging in $\lambda_{k,\infty}(t)$ into \cref{eqs:lambda} yields the following bound. 

\begin{corollary}[without correction]  \label{cor:ineq-uncorrected}
For $t>k-1$, it holds that
\begin{equation}
\P\left( n \D(\hat{p}_{k,n} \| p) > t \right) \leq e^{-t} e^{k-1} G_{k,n}\left(1 - \frac{k-1}{t} \right).
\end{equation}
\end{corollary}

$\lambda_{k,\infty}(t)$ is the zeroth-order large $n$ approximation to $\lambda_{k,n}(t)$. Yet, the bound can be significantly tightened by a further correction.

\begin{proposition} \label{prop:correction}
Suppose $k \geq 2$ and $t > k-1$. As $n \rightarrow \infty$, we have
\begin{equation}
\lambda_{k,n}(t) = \lambda_{k,\infty}(t) + \frac{k}{k-1} \frac{t - k+1}{n} + o(n^{-1}).
\end{equation}
\end{proposition}
\begin{proof}
Fix $k \geq 2$ and $t > k-1$. Let $f_{k,n}:= \exp(-\lambda t) G_{k,n}(\lambda)$. First, we claim that there exists $N(k,t)$ such that $f_{k,n}'(\lambda_{k,n}) = 0$ for $n \geq N(k,t)$ at the minimizer $\lambda_{k,n}$. To see this, note that asymptotically $\lambda_{k,n}$ cannot be 0 or 1. In particular, (i) $\lambda_{k,n} = 0$ would imply $\text{RHS}=1$ for \cref{eqs:chernoff}, and (ii) $\lambda_{k,n} \rightarrow 1$ would imply $\text{RHS} \rightarrow \infty$ for \cref{eqs:chernoff} --- both contradict \cref{prop:meaningful}. Given 
\begin{equation*}
f'_{k,n}(\lambda_{k,n}) = f'_{k,n}(\lambda_{k,\infty}) + f''_{k,n}(\lambda_{k,\infty})(\lambda_{k,n} - \lambda_{k,\infty}) + o(|\lambda_{k,n} - \lambda_{k,\infty}|),
\end{equation*}
it follows that
\begin{equation*}
\lambda_{k,n} = \lambda_{k,\infty} - \frac{f'_{k,n}(\lambda_{k,\infty})}{f''_{k,n}(\lambda_{k,\infty})} + o(|\lambda_{k,n} - \lambda_{k,\infty}|).
\end{equation*}
Since $f_{k,n} \rightarrow f_{k,\infty} = \exp(-\lambda t) G_{k,\infty}(\lambda)$, it is easy to check that
\begin{equation*}
\begin{split}
f''_{k,n}(\lambda_{k,\infty}) &= f''_{k,\infty}(\lambda_{k,\infty}) + o(1) \\
&= (k-1) e^{-\lambda_{k,\infty} t} (1-\lambda_{k,\infty})^{-(k+1)} + o(1),
\end{split}
\end{equation*}
where the limit $(k-1) e^{-\lambda_{k,\infty} t} (1-\lambda_{k,\infty})^{-(k+1)}$ is non-zero and finite. Meanwhile, we have
\begin{equation*}
f'_{k,n}(\lambda_{k,\infty}) = e^{-\lambda_{k,\infty} t} \left(G'_{k,n}(\lambda_{k,\infty}) - t G_{k,n}(\lambda_{k,\infty}) \right).
\end{equation*}
Using $\lambda_{k,\infty} = 1 - (k-1)/t$, it follows that
\begin{equation*}
\lambda_{k,n} = \lambda_{k,\infty} + \frac{e^{-(t-k+1)} \left[\frac{k-1}{1-\lambda_{k,\infty}} G_{k,n}(\lambda_{k,\infty}) - G'_{k,n}(\lambda_{k,\infty})\right]}{(k-1) e^{-(t-k+1)} \left(\frac{k-1}{t} \right)^{-(k+1)} + o(1)} + o(|\lambda_{k,n} - \lambda_{k,\infty}|).
\end{equation*}
It is easy to check that the proof is complete given the following lemma.
\end{proof}

\begin{lemma} \label{lem:limit-correction}
For $k \geq 2$ and $\lambda \in (0,1)$, it holds that
\begin{equation} \label{eqs:limit-correction}
n \left( \frac{k-1}{1-\lambda} G_{k,n}(\lambda) - G'_{k,n}(\lambda) \right) \rightarrow \frac{k(k-1)\lambda}{(1-\lambda)^{k+2}}, \quad \text{as $n \rightarrow \infty$}.
\end{equation}
\end{lemma}
The proof relies on asymptotic expansions of the incomplete gamma function and is left to the Appendix.

\begin{remark}
The correction in \cref{prop:correction} can be viewed as a one-step Newton's iteration based on $\lambda_{k,\infty}(t)$.
\end{remark}

\begin{figure}[!htb]
\centering
\includegraphics[width=0.45\textwidth]{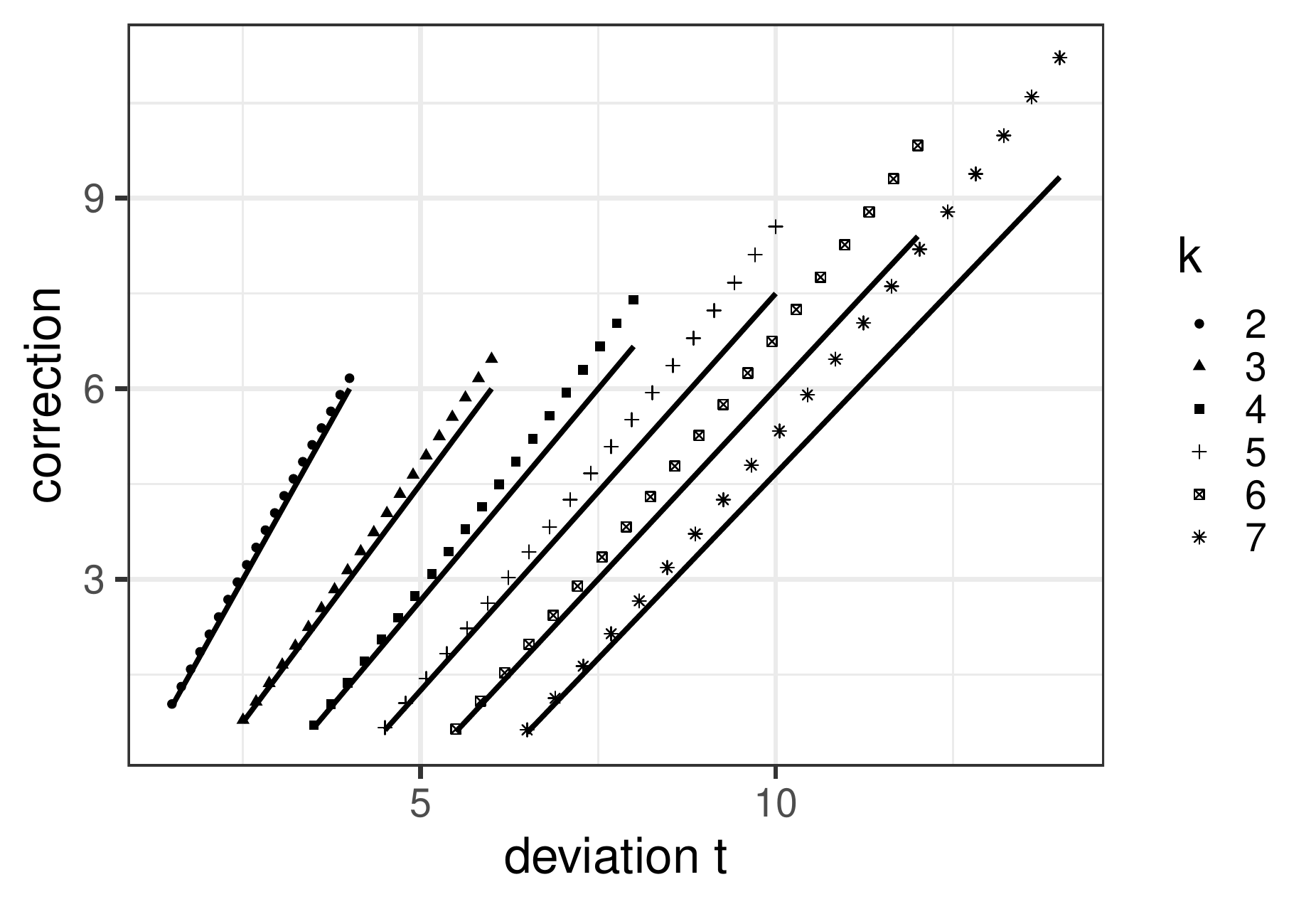}
\caption{The ideal correction $\lim_{n \rightarrow \infty} n(\lambda_{k,n}(t) - \lambda_{k,\infty}(t))$ (dots, fitted from numerical values) and the theoretical first-order correction  $k(t - k+1) / (k-1)$ (lines), both plotted against the deviation $t$.}
\label{fig:correction}
\end{figure}

In \cref{fig:correction}, we compare the correction term (the $n^{-1}$ term) from \cref{prop:correction} to the numerical values. The numerical value corresponding to a pair $(t, k)$ is obtained by numerically finding $\lambda_{k,n}(t)$ for a sequence of $n$ varying from $200$ to $2 \times 10^4$, then fitting $\log(\lambda_{k,n} - \lambda_{k,\infty})$ against $- \log n$ in least squares, and finally taking the intercept and exponentiating. 

Plugging the correction into \cref{eqs:lambda} yields the following bound. 

\begin{corollary}[with correction] \label{cor:bound-correction}
Let $\hat{\lambda}_{k,n} := \min\left\{1 - \frac{k-1}{t} + \frac{k}{k-1} \frac{t - k+1}{n}, 1 \right\}$. For $n \geq 1$, $k \geq 2$ and $t > k-1$, it holds that
\begin{equation}
\P\left( n \D(\hat{p}_{k,n} \| p) > t \right) \leq \exp(-\hat{\lambda}_{k,n} t) G_{k,n}(\hat{\lambda}_{k,n}).
\end{equation}
\end{corollary}

\section{Discussion} In this section, we discuss the behavior of our bound and compare to bounds previously proposed in the literature. 

\subsection{Comparison} 
We briefly compare the bounds for several sample sizes under $k=6$ in \Cref{fig:probs}; see \cref{fig:probs-ii} in the Appendix for $k=20$. First, our bound is always tighter than \citet{agrawal2019concentration}, since \citet{agrawal2019concentration} uses Chernoff bound based on $G_{k,\infty}$, which upper-bounds $G_{k,n}$. Second, in the settings plotted, our bound is tighter than that of \citet{mardia2018concentration} for $t$ smaller than some $T_{k,n}$ and vice versa for $t > T_{k,n}$ --- an explanation for this phenomenon is provided in the following section. Third, the closed-form correction-based bound is significantly tighter than the bound without correction, and is in fact very close to the exact bound, with the difference between the two only noticeable when both $n$ and $t$ are small.

\begin{figure*}[!ht]
\centering
\includegraphics[width=.9\textwidth]{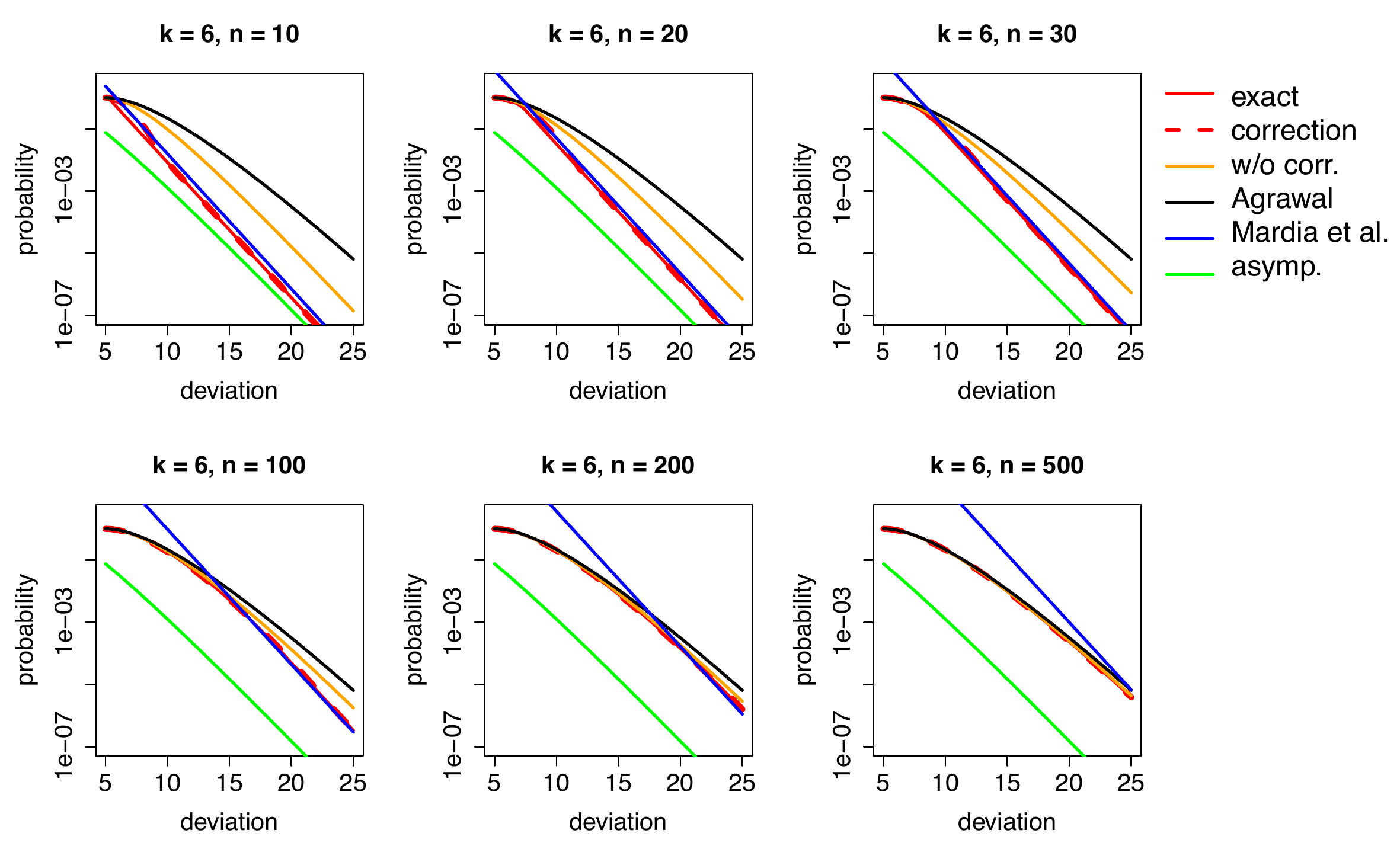}
\caption{Comparison of probability bounds on $\P(n \D(\hat{p}_{k,n} \| p) > t)$ for $k=6$ and $t > \min(\log G_{k,n}(1), k-1)$. The $y$-axis is in logarithmic scale. The methods compared include: ``exact'' (\cref{thm:chernoff} from numerical minimization), ``correction'' (\cref{cor:bound-correction}), ``w/o corr.'' (\cref{cor:ineq-uncorrected}), \citet[Theorem 1.2]{agrawal2019concentration}, \citet[Theorem 3]{mardia2018concentration}, and the asymptotic bound that is the exact probability when $n \rightarrow \infty$. Note that ``asymp.'' might not be a valid bound and is for reference only.}
\label{fig:probs}
\end{figure*}

\subsection{Combinatorial scaling}
Recently \citet{mardia2018concentration} consider a bound of the form
\begin{equation} \label{eqs:combinatorial-ineq}
\P\left(n \D(\hat{p}_{k,n} \| p) > t \right) \leq C(k,n) \exp(-t),
\end{equation}
where $C(k,n)$ captures the combinatorial dependence on $k$ and $n$. This is motivated by the classic method-of-types inequality \cref{eqs:MoT}, which holds with
\begin{equation*}
C_{\mathrm{T}}(k,n) = \binom{n + k - 1}{k - 1}.
\end{equation*}
Note that $C_{\mathrm{T}}(k,n)$ is the number of ways that $\{1,\dots,n\}$ can be partitioned into $k$ groups, and hence counts the ``types'' of possible empirical distributions. 
\citet{mardia2018concentration} showed that $C_{\mathrm{T}}(k,n)$ can be improved to 
\begin{equation*}
C_{\mathrm{M}}(k, n) = \frac{12}{\pi} \sum_{i=0}^{k-2} K_{i-1} \left(\frac{e \sqrt{n}}{2 \pi} \right)^i,
\end{equation*}
where 
\begin{equation*}
K_i = \begin{cases} \frac{\pi (2\pi)^{i/2}}{2 \times 4 \times \dots \times i} & \  \text{($i$ is even)} \\
\frac{(2\pi)^{(i+1)/2}}{1 \times 3 \times \dots \times i} & \  \text{($i$ is odd)}
\end{cases}, \quad K_{-1} = 1
\end{equation*}
are constants. It can be shown that $C_{\mathrm{M}}(k, n)$ is smaller than $C_{\mathrm{T}}(k,n)$ for all $k, n \geq 2$ \citep[\S1.2]{mardia2018concentration}

Since the choice of $\lambda$ that tightens our bound depends on $t$, the bounds presented in the previous section do not take the form of \cref{eqs:combinatorial-ineq}. For comparison, we use the following bound from setting $\lambda = 1$ in \cref{eqs:lambda}, which is not the tightest bound except for very large $t$. 
\begin{corollary} \label{cor:combinatorial}
For $n \geq 1$, $k \geq 2$ and $t > 0$, it holds that 
\begin{equation*}
\P\left(n \D(\hat{p}_{k,n} \| p) > t \right) \leq G_{k,n}(1) \exp(-t).
\end{equation*}
\end{corollary}

Like $C_{\mathrm{M}}(k,n)$ the resulting combinatorial factor $G_{k,n}(1)$ is also uniformly smaller than the method-of-types combinatorial factor $C_{\mathrm{T}}(k,n)$. 
\begin{proposition}
For $k \geq 2$, $n \geq 1$, $G_{k,n}(1) < C_{\mathrm{T}}(k,n)$. 
\end{proposition}
\begin{proof}
By \cref{thm:polynomial}, 
\begin{equation*}
\begin{split}
\quad G_{k,n}(1) &= \sum_{m=0}^{n} \frac{n \times (n-1) \times \dots \times (n-m+1)}{n^m} \binom{m+k-2}{k-2} \\
&< \sum_{m=0}^{n} \binom{m+k-2}{k-2} = \binom{n+k-1}{k-1},
\end{split}
\end{equation*}
where the last equality follows from the ``parallel summation'' \citep[Eq.~(5.9)]{gkp-book}.
\end{proof}

In fact, the improvement can be significant when $n$ is large.
\begin{proposition} \label{prop:sqrt-improvement}
For fixed $k \geq 2$, as $n \rightarrow \infty$, $\frac{\log G_{k,n}(1)}{\log C_{\mathrm{
T}}(k,n)} \rightarrow 1/2$.
\end{proposition}
This basically says, in the regime of fixed $k$ and large $n$, $G_{k,n}(1)$ is a square-root improvement over the method-of-types combinatorial factor. We leave its proof to \cref{apx:sqrt-improvement}. In fact, $C_{\mathrm{M}}(k,n)$ achieves the same rate of improvement in the same regime; see \citet[\S 1.2]{mardia2018concentration}. For other regimes, we do not have an explicit comparison. Instead, in \cref{fig:combina} we graphically compare the combinatorial factors for a few $(k,n)$. We observe: (i) $\log G_{k,n}(1)$ and $\log C_{\mathrm{M}}(k,n)$ scale quite closely; (ii) for a fixed $k$, one can check that $G_{k,n}(1) < C_{\mathrm{M}}(k,n)$ for small $n$, and vice versa for large $n$. Note that (ii) explains why in \cref{fig:probs} the bound of \citet{mardia2018concentration} becomes tighter than our bound for very large deviations when $n \in \{100, 200, 500\}$ --- the tightening $\lambda_{k,n}(t)=1$ for $t$ large enough and the exact bound reduces to \cref{cor:combinatorial}.

\begin{figure}[!ht]
\centering
\includegraphics[width=0.35\textwidth]{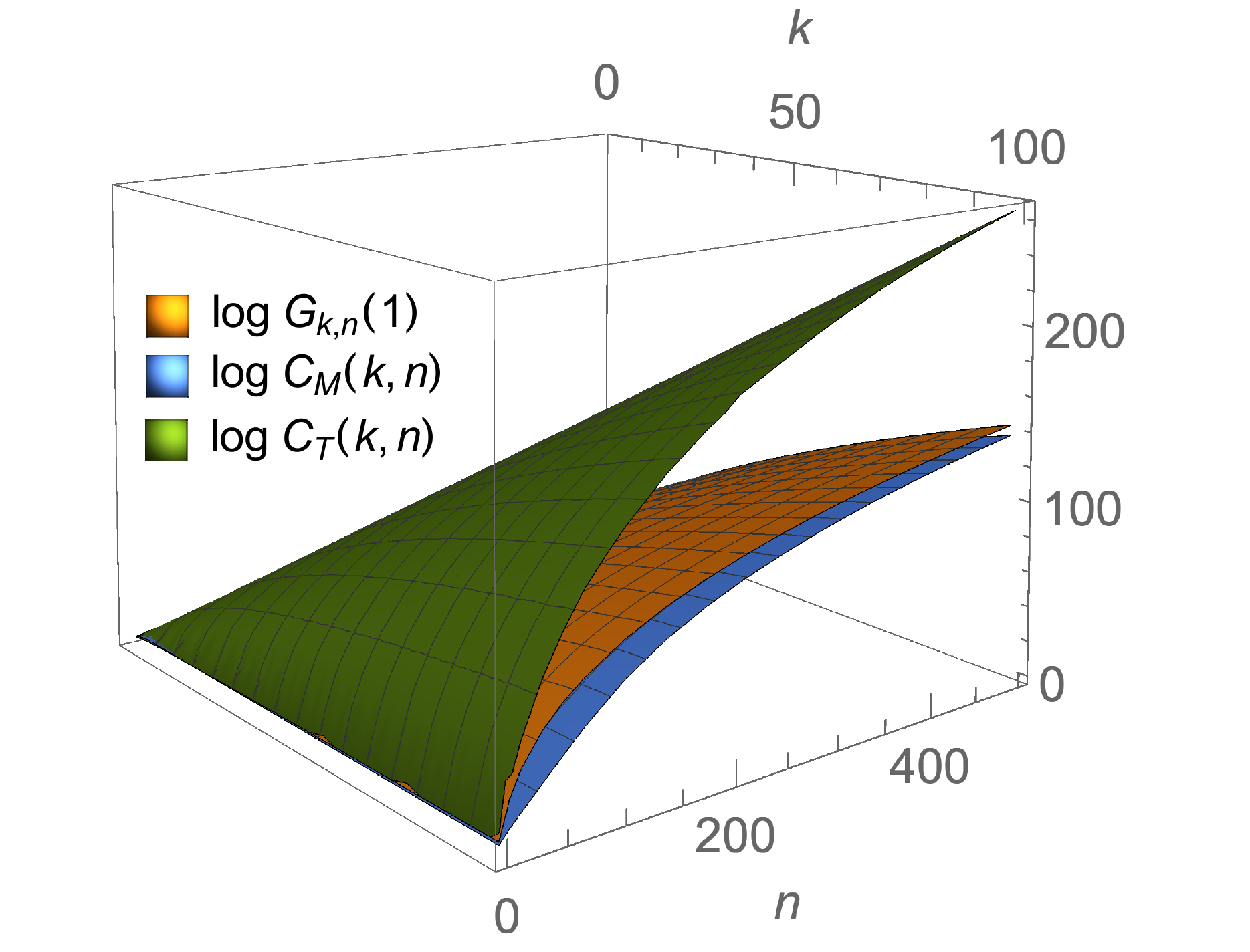}
\includegraphics[width=0.44\textwidth]{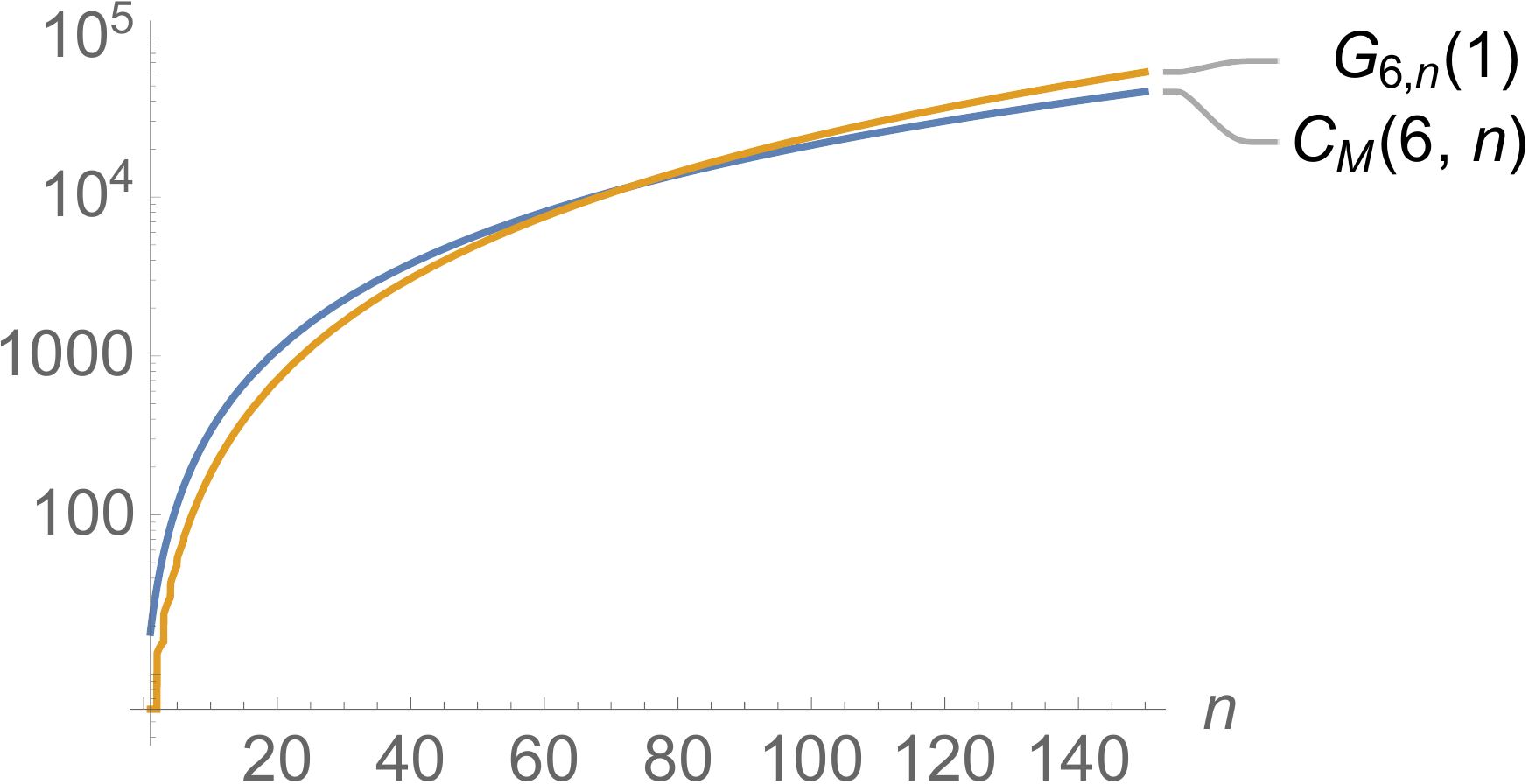}
\caption{Comparison of combinatorial scaling factors $G_{k,n}(1)$ (ours), $C_{\mathrm{M}}(k,n)$ \citep{mardia2018concentration} and $C_{\mathrm{T}}(k,n)$ (method of types).}
\label{fig:combina}
\end{figure}

Finally, we stress that the improved combinatorial factors are by no means optimal. To see this, note that as $n \rightarrow \infty$, $G_{k,n}(1) \rightarrow \infty$ for any fixed $k \geq 2$ and $C_{\mathrm{M}}(k,n) \rightarrow \infty $ for any fixed $k \geq 3$, which would render the bound in the form of \cref{eqs:combinatorial-ineq} meaningless (for fixed $k$ and $t$). However, by \cref{prop:large-n-limit} because $G_{k,\infty}(\lambda)$ only diverges at $\lambda=1$, our bounds stated in \cref{thm:polynomial}, \cref{cor:ineq-uncorrected,cor:bound-correction} do not suffer from this problem. Nevertheless, we expect future improvements on $C(k,n)$ such that $C(k,\infty) < \infty$ for $k \geq 2$.

 \section{Application}
The bound developed can be used to obtain a conservative non-asymptotic critical value for the multinomial likelihood ratio. The bound in \cref{thm:chernoff} can be determined numerically by searching for the minimizer over the unit interval, which is a non-convex but smooth, univariate optimization. Further given a level $\alpha \in (0,1)$ (e.g., $\alpha = 0.05$), by a binary search, a critical value $t_{k,n}(\alpha)$ can be determined such that the bound at $t_{k,n}(\alpha)$ evaluates to $\alpha$. The critical value on the likelihood ratio can be inverted to form a convex confidence region on $p$, which is guaranteed to contain $p$ with probability at least $(1-\alpha)$. This can be applied to the cases where $k$ is comparable to $n$, and the standard large-sample $\chi^2$ approximation is unlikely to be accurate (see \citet{frydenberg1989improved}). We demonstrate with the following example. 

\paragraph{Proportion of the unseen butterflies}
\Cref{tab:butterfly} shows the famous dataset \citep{orlitsky2016optimal} that naturalist Alexander Steven Corbet presented to Ronald Fisher in the 1940's. Corbet spent two years trapping butterflies in Peninsular Malaysia, and his intriguing question to Fisher was how many new species would he discover had he spent another two years on the islands. Corbet's original question led to the fruitful investigation of estimating the number of unseen species; see \citet{fisher1943relation,good1956number,orlitsky2016optimal}. 

\begin{table*}[!htb]
\caption{Butterflies recorded by Corbet}
\centering
\begin{tabular}{llllllllllllllll}
Frequency & 1   & 2  & 3  & 4  & 5  & 6  & 7  & 8  & 9  & 10 & 11 & 12 & 13 & 14 & 15 \\
Species   & 118 & 74 & 44 & 24 & 29 & 22 & 20 & 19 & 20 & 15 & 12 & 14 & 6  & 12 & 6 
\end{tabular}
\label{tab:butterfly}
\end{table*}

However, here we pose a different question --- what percentage of butterflies in Malaya belonged to the species that Corbet had not seen? That is, we want to estimate the proportion of butterflies from all the unseen species. Clearly, the MLE is zero based on the sample. Instead, we ask for an upper bound with $95\%$ confidence. Let $k = 435 + 1$, where 435 is the number of species observed by Corbet. Let $\hat{p} = (\hat{q}, 0)$, where $\hat{q}$ is the empirical distribution corresponding to \cref{tab:butterfly}. The sample size is $n = 2,029$ and the corresponding critical value is $t_{k,n}(\alpha) = 481.20 $. The upper bound is given by the convex program
\begin{equation*}
\max p_{k} \quad \text{s.t.  } p \in \Delta^{k-1},\ n \D(\hat{p} \| p) \leq t_{k,n}(\alpha),
\end{equation*}
which evaluates to $21.1\%$. See also \citet{robbins1968estimating,bickel1986estimating} related to this problem. 

 \section{Conclusion}
We have shown that for a multinomial experiment with alphabet size $k$ and sample size $n$, the moment generating function of the entropy of the empirical distribution relative to the true distribution (scaled by $n$) can be uniformly bounded by a degree-$n$ polynomial $G_{k,n}(\lambda)$ over the unit interval. We generalize Agrawal's result \citep{agrawal2019concentration} on $k=2$ and characterize the family of $G_{k,n}(\lambda)$. The result gives rise to a one-sided Chernoff bound on the relative entropy for deviations $t > \min(\log G_{k,n}(1), k-1)$. The bound significantly improves the classic method-of-types bound and is competitive with the state of the art \citep{mardia2018concentration}.  Further, since the tightest Chernoff bound does not permit a closed-form, we have developed a first-order large-$n$ expansion of the minimizing $\lambda$, which provides a good approximation to the tightest bound in closed form. On a technical note, our approach directly constructs bounds for a generic $k$, in contrast to some other approaches \citep{mardia2018concentration,agrawal2019concentration} that are based on a reduction from multinomial to binomial via the chain rule of relative entropy.   \appendix
\section{Proofs}
\subsection{Proof of \cref{prop:recurrence}}
\begin{proof}
By \cref{prop:p-independent}, $G_{k,n}(\lambda) = G_{k,n}(\lambda, p)$ for $p_k=0$ and $p_1 + \dots + p_{k-1} = 1$. By \cref{eqs:G-kn}, we split $G_{k,n}(\lambda) = A + B$, where $A$ sums over those $X$ with $X_k=0$, and $B$ sums over those with $X_k \geq 1$. Clearly, 
\begin{equation*}
A = \sum_{X_1, \dots, X_{k-1}} \binom{n}{X_1, \dots, X_{k-1}} \prod_{j=1}^{k-1} \left[\lambda X_j / n + (1-\lambda) p_j \right]^{X_j},
\end{equation*}
where the summation is over non-negative integers $X_1, \dots, X_{k-1}$ such that they sum to $n$. Further, $(p_1, \dots, p_{k-1})$ forms a probability vector. Hence, $A = G_{k-1,n}(\lambda)$. 

Now we evaluate 
\begin{equation*}
B = \sum_{X_k=1}^{n} \sum_{X_1+\dots+X_{k-1}=n-X_k} \binom{n}{X_1,\dots,X_k} \prod_{j=1}^{k} \left[\lambda X_j / n + (1-\lambda) p_j \right]^{X_j}.
\end{equation*}
Using the fact that $\binom{n}{X_1,\dots,X_k} = \left(\frac{n}{X_k}\right) \binom{n-1}{X_1, \dots, X_{k-1},X_k-1}$ and $p_k = 0$, we have
\begin{equation} \label{eqs:B-expr}
\begin{split}
B &= \sum_{X_1, \dots, X_{k-1}, X_k'} \frac{n}{X_k'+1} \binom{n-1}{X_1, \dots, X_{k-1}, X_k'}  \left(\frac{\lambda(X_k'+1)}{n} \right)^{X_k'+1} \prod_{j=1}^{k-1} \left[\lambda X_j / n + (1-\lambda) p_j \right]^{X_j} \\
&= \lambda \sum_{X_1, \dots, X_{k-1}, X_k'}  \binom{n-1}{X_1, \dots, X_{k-1}, X_k'}  \left(\frac{\lambda(X_k'+1)}{n} \right)^{X_k'} \prod_{j=1}^{k-1} \left[\lambda X_j / n + (1-\lambda) p_j \right]^{X_j},
\end{split}
\end{equation}
where $X_k':=X_k-1 \in \{0,\dots,n-1\}$ and the summation is over $(X_1, \dots, X_{k-1}, X_k')$ such that they sum to $n-1$. Let $\lambda' := \frac{n-1}{n} \lambda$ and 
\begin{equation*}
p_j' := \frac{1-\lambda}{1-\lambda'} p_j \quad (j=1,\dots,k-1), \quad p_k':= \frac{\lambda/n}{1-\lambda'}
\end{equation*}
such that $\sum_{j=1}^{k} p_j' = \frac{1-\lambda}{1-\lambda'} + \frac{\lambda/n}{1-\lambda'} = 1$. Then we have 
\begin{equation*}
\frac{\lambda(X_k'+1)}{n} = \lambda' \frac{X_k'}{n-1} + (1-\lambda') p_k',
\end{equation*}
and 
\begin{equation*}
\lambda X_j/n + (1-\lambda)p_j = \lambda' \frac{X_j}{n-1} + (1-\lambda') p_j' \quad (j=1,\dots,k-1).
\end{equation*}
Hence, by \cref{eqs:G-kn} and \cref{prop:p-independent}, \cref{eqs:B-expr} becomes
\begin{equation*}
\begin{split}
B &= \lambda \sum_{X_1, \dots, X_{k-1}, X_k'}  \binom{n-1}{X_1, \dots, X_{k-1}, X_k'} \prod_{j=1}^{k} \left[\lambda' X_j / {n-1} + (1-\lambda') p_j' \right]^{X_j}  \\
&= \lambda G_{k,n-1}(\lambda') = \lambda G_{k, n-1}\left(\frac{n-1}{n} \lambda \right).
\end{split}
\end{equation*}
Putting $A$ and $B$ together, we have $G_{k,n}(\lambda) = G_{k-1,n}(\lambda) + \lambda G_{k,n-1} \left(\frac{n-1}{n} \lambda \right)$. 
\end{proof}

\subsection{Proof of \cref{lem:limit-correction}}
We will use the following two properties of the incomplete gamma function 
\begin{equation*}
\Gamma(a, z) := \int_{z}^{\infty} t^{a-1} e^{-t} \dd t.
\end{equation*}

\begin{lemma}[{\citet[\S 8.8]{NIST:DLMF}}] \label{lem:gamma-recurrence}
It holds that
\begin{equation*}
\Gamma(a+1, z) = a \Gamma(a, z) + z^{a} e^{-z},
\end{equation*}
and 
\begin{equation*}
\Gamma(a, z) = \frac{\Gamma(a)}{\Gamma(a-n)} \Gamma(a-n, z) + z^{a-1} e^{-z} \sum_{k=0}^{n-1} \frac{\Gamma(a)}{\Gamma(a-k)} z^{-k},
\end{equation*}
where $n$ is a non-negative integer. 
\end{lemma}

\begin{lemma}[{\citet[\S 8.11(iii)]{NIST:DLMF}}] \label{lem:gamma-asymp}
For fixed $\gamma > 1$, as $a \rightarrow \infty$, it holds that
\begin{equation*}
\Gamma(a, \gamma a) = (\gamma a)^{a} e^{-\gamma a} \left\{ \sum_{k=0}^{n} \frac{(-1)^{k} b_k(\gamma)}{(\gamma - 1)^{2k+1}} a^{-k-1} + o(|a|^{-n-1}) \right\},
\end{equation*}
where $b_0(\gamma) = 1$, $b_1(\gamma) = \gamma$, $b_2(\gamma) = \gamma (2\gamma +1)$, and for $k=1,2,\dots$, 
\begin{equation*}
b_k(\gamma) = \gamma(1-\gamma) b'_{k-1}(\gamma) + (2k-1) \gamma b_{k-1}(\gamma).
\end{equation*}
\end{lemma}

\begin{proof}[Proof of \cref{lem:limit-correction}]
We first express $G_{2,n}(\lambda)$ in terms of the incomplete gamma function:
\begin{equation} \label{eqs:G-Gamma-form}
\begin{split}
G_{2,n}(\lambda) &= \sum_{m=0}^{n} \frac{n!}{n^m (n-m)!} \lambda^m \\
&= \lambda^n \sum_{m=0}^{n} \frac{n!}{n^m (n-m)!} \lambda^{-(n-m)} \\
&= n^{-n} \lambda^{n} n! \sum_{m=0}^{n} \frac{(n/\lambda)^m}{m!} = n^{-n} \lambda^n e^{n / \lambda} \Gamma(n+1, n / \lambda),
\end{split}
\end{equation}
where we used the fact \citep[Eq.~8.4.8]{NIST:DLMF} that 
\begin{equation*}
\Gamma(n+1, z) = n! e^{-z} \sum_{k=0}^{n} \frac{z^k}{k!}, \quad n=0,1,2,\dots.
\end{equation*}
Using \cref{lem:gamma-recurrence,lem:gamma-asymp}, as $n \rightarrow \infty$, we have 
\begin{equation*}
\begin{split}
& \quad G_{2,n}(\lambda) \\
&= (\lambda / n)^n e^{n / \lambda} \left\{ n \Gamma(n, n/\lambda) + (n/\lambda)^n e^{-n/\lambda} \right\} \\
&= 1 + (\lambda / n)^n e^{n / \lambda} n \Gamma(n, n/\lambda) \\
&= 1 + (\lambda / n)^n e^{n / \lambda} n  (n/\lambda)^{n} e^{-n/\lambda} \left\{ \frac{n^{-1}}{\lambda^{-1} - 1} - \frac{\lambda^{-1}}{(\lambda^{-1} - 1)^3} n^{-2} + O(n^{-3}) \right \} \\
&= \frac{1}{1-\lambda} - \frac{\lambda^2}{(1-\lambda)^3} n^{-1} + O(n^{-2}).
\end{split}
\end{equation*}
For $k \geq 2$, it follows from \cref{eqs:Gkn-deriv-form} that 
\begin{equation*}
\begin{split}
& \quad G_{k,n}(\lambda) \\
&= \frac{1}{(k-2)!} \frac{\dd^{k-2}}{\dd \lambda^{k-2}} \left( \frac{\lambda^{k-2}}{1-\lambda} - \frac{\lambda^{k}}{(1-\lambda)^3} n^{-1} + O(n^{-2}) \right) \\
&= (1-\lambda)^{-(k-1)} - \frac{k(k-1)\lambda^2}{2(1-\lambda)^{k+1}} n^{-1} + O(n^{-2}),
\end{split}
\end{equation*}
where we used the relation (twice)
\begin{equation*}
\frac{\dd^{k-2}}{\dd \lambda^{k-2}} \left[\frac{\lambda^{k}}{(1-\lambda)^3} \right] = \frac{k! \lambda^2}{2 (1-\lambda)^{k+1}},
\end{equation*}
which can be shown via induction. 

Finally, we have 
\begin{equation*}
\begin{split}
& \quad \frac{k-1}{1-\lambda} G_{k,n}(\lambda) - G'_{k,n}(\lambda) \\
&= \frac{1}{(1-\lambda)^{k-1}} \frac{\dd}{\dd \lambda} \left[-(1-\lambda)^{k-1} G_{k,n}(\lambda) \right] \\
&= \frac{1}{(1-\lambda)^{k-1}}  \frac{\dd}{\dd \lambda} \bigg[-(1-\lambda)^{k-1} \bigg( (1-\lambda)^{-(k-1)} - \frac{k(k-1)\lambda^2}{2(1-\lambda)^{k+1}} n^{-1} + O(n^{-2}) \bigg) \bigg] \\
&= \frac{1}{(1-\lambda)^{k-1}}  \frac{\dd}{\dd \lambda} \left[-1 + \frac{k(k-1)\lambda^2}{2(1-\lambda)^2} n^{-1} + O(n^{-2}) \right] \\ 
&= \frac{k(k-1)\lambda}{(1-\lambda)^{k+2}} n^{-1} + O(n^{-2}),
\end{split}
\end{equation*}
and hence 
\begin{equation*}
n \left( \frac{k-1}{1-\lambda} G_{k,n}(\lambda) - G'_{k,n}(\lambda) \right) \rightarrow \frac{k(k-1)\lambda}{(1-\lambda)^{k+2}}.
\end{equation*}
\end{proof}

\subsection{Proof of \cref{prop:sqrt-improvement}} \label{apx:sqrt-improvement}
\begin{lemma}[{\citet[\S 8.11(v)]{NIST:DLMF}}] \label{lem:Gamma-z-z-asymp}
As $z \rightarrow \infty$, it holds that
\begin{equation*}
\Gamma(z,z) = z^{z-1} e^{-z} \left(\sqrt{\frac{\pi}{2}} z^{1/2} + O(1) \right).
\end{equation*}
\end{lemma}

\begin{proof}[Proof of \cref{prop:sqrt-improvement}]
For $k=2$, we have
\begin{equation*}
\begin{split}
G_{2,n}(1) &\stackrel{\mathrm{(i)}}{=} (e/n)^{n} \Gamma(n+1, n) \\
& \stackrel{\mathrm{(ii)}}{=} (e/n)^{n}n \Gamma(n,n) + 1 \\
&\stackrel{\mathrm{(iii)}}{=} \sqrt{\frac{\pi}{2}} n^{1/2} + O(1),
\end{split}
\end{equation*}
where (i) follows from \cref{eqs:G-Gamma-form}, (ii) from \cref{lem:gamma-recurrence} and (iii) from \cref{lem:Gamma-z-z-asymp}. And hence, 
\begin{equation*}
\lim_{n \rightarrow \infty} \frac{\log G_{2,n}(1)}{\log \binom{n+1}{1}} = \lim_{n \rightarrow \infty} \frac{\log n^{1/2}}{\log n} = 1/2.
\end{equation*}
By a similar computation for $k=3,4,\dots$, one can show that
\begin{equation*}
\lim_{n \rightarrow \infty} \frac{\log G_{k,n}(1)}{\log \binom{n+k-1}{k-1}} = \lim_{n \rightarrow \infty} \frac{\log n^{(k-1)/2}}{\log n^{k-1}} = 1/2.
\end{equation*}
\end{proof}

\section{Additional plot}
\begin{figure*}[!h]
\centering
\includegraphics[width=.9\textwidth]{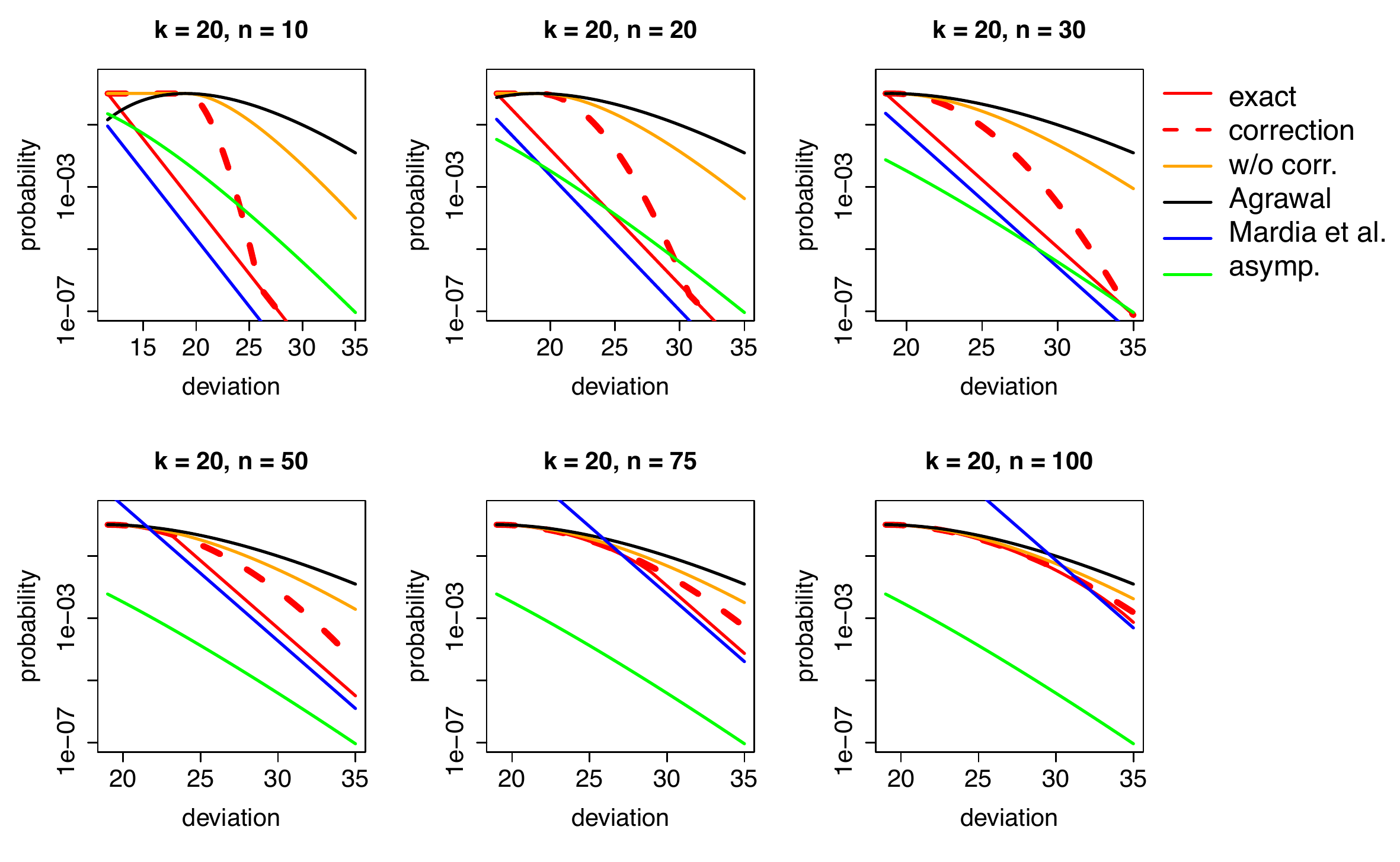}
\caption{Comparison of probability bounds on $\P(n \D(\hat{p}_{k,n} \| p) > t)$ for $k=20$ and $t > \min(\log G_{k,n}(1), k-1)$. The $y$-axis is in logarithmic scale. The methods compared include: ``exact'' (\cref{thm:chernoff} from numerical minimization), ``correction'' (\cref{cor:bound-correction}), ``w/o corr.'' (\cref{cor:ineq-uncorrected}), \citet[Theorem 1.2]{agrawal2019concentration}, \citet[Theorem 3]{mardia2018concentration}, and the asymptotic bound that is the exact probability when $n \rightarrow \infty$. Note that ``asymp.'' might not be a valid bound and is for reference only.}
\label{fig:probs-ii}
\end{figure*} 
\section*{Acknowledgment}
F.~Richard Guo thanks Jon Wellner for helpful comments. The research was supported by ONR Grant N000141912446.

\end{document}